\documentclass{article}
\usepackage{amsmath,amssymb,amsfonts}
\usepackage{amsthm}
\usepackage{latexsym}
\usepackage{bbold} 
\newcommand{\R}{{\mathbb R}}

\newcommand{\N}{{\mathbb N}}

\newcommand{\non}{{\nonumber}}

\newcommand{\W}{{\mathcal W}}

\newcommand{\e}{\varepsilon}

\renewcommand{\d}{{\rm d}}

\newcommand{\wu}{\underline{w}}

\newcommand{\vu}{\underline{u}}

\newcommand{\he}{\lambda}
\newcommand{\rel}{\rho^{\rm el.}_\he}
\newcommand{\Id}{{\rm Id}}

\newtheorem{theorem}{Theorem}
\newtheorem{lemma}{Lemma}
\newtheorem{cor}{Corollary}
\newtheorem{prop}{Proposition}
\theoremstyle{definition}
\newtheorem{definition}{Definition}
\newtheorem{rem}{Remark}
\usepackage[utf8]{inputenc}
\usepackage{authblk}
\title{Almost conical deformations of thin sheets with
rotational symmetry}
\author[1]{Stefan M\"uller\thanks{stefan.mueller@hcm.uni-bonn.de}}
\author[1]{Heiner Olbermann\thanks{heiner.olbermann@hcm.uni-bonn.de}}
\affil[1]{Hausdorff Center for Mathematics \& Institute for Applied Mathematics,
 University of Bonn, Germany}
\date{\today}
\begin{document}
\maketitle
\begin{abstract}
\noindent
It has been found in numerical experiments \cite{2006PhRvE..73d6604L} that when
one removes a sector from an elastic sheet and glues the edges of the sector
back together, the resulting configuration is radially symmetric  and nearly
conical. We make a rigorous analysis of this setting under two simplyfying
assumptions: Firstly, we only consider radially symmetric
configurations. Secondly, we consider the so-called von-K\'arm\'an limit,
where the size of the removed region as well as the deformations are small. We
choose free boundary conditions for a sheet of infinite size.
% The arising minimization problem  is a free boundary problem, as we do not specify boundary conditions at the outer radius $R$ (which we choose to be $\infty$).
We show existence  of  minimizers of the suitably renormalized free energy
functional. As a by-product, we obtain a lower bound for the elastic energy
that has been conjectured in the related context of \emph{d-cones}
\cite{2012arXiv1208.4298M}. Moreover, we determine the  shape of minimizers at
infinity up to  terms that decay like $\exp(-c\sqrt{r})$.
\end{abstract}
\section{Introduction}

% {\bf SM 13.7.13: still to do}\\
% Dependence on $\lambda$\\
% Eigenvalue calculation (double checl)\\
% Check proof of stable mf. thm.\\
% Optimal interpolation estimate for $e^{t/2} g$ ?\\
% Read whole paper\\

The folding of paper or other thin elastic sheets is one of the many examples of energy focusing in the physical world. % E.g., in the guise of turbulence, energy focusing is notoriously difficult to model and handle even in a non-rigorous setting. Paper folding seems to be much more amenable to a rigorous mathematical analysis.
Starting in the late 90's, there has been a lot of interest in this problem in the physics
community
\cite{PhysRevE.71.016612,PhysRevLett.80.2358,Cerda08032005,RevModPhys.79.643,CCMM,PhysRevLett.87.206105,PhysRevLett.78.1303,Lobkovsky01121995,PhysRevLett.90.074302}. In
particular the crumpling of paper (i.e.~the crushing of a thin elastic sheet
into a container whose diameter is smaller than the size of the sheet) which
results in complex folding patterns has drawn a lot of attention. It has been
conjectured that the energy density per thickness $h$ of such a folding pattern
scales with $h^{5/3}$. One major contribution in the rigorous analysis of this
problem is \cite{MR2358334}, building on ideas from \cite{MR2023444}.\\
Here we focus on  approximately conical deformations of thin elastic
sheets, that can be viewed as (one kind of) building blocks of crumpled
deformations. One example for this is a sheet that is pushed into a hollow cylinder, such that the indentation of the sheet is small. % (This experimental set-up can be realized even with the restricted amount of equipment available to mathematicians: a sheet of paper, a water glass and a pencil, see fig.[...].)
The resulting structure is called a \emph{d-cone} (developable cone). In the
physics literature, it has been discussed e.g.~in
\cite{PhysRevLett.80.2358,Cerda08032005,PhysRevE.71.016612,RevModPhys.79.643}. There
are several remarkable features of the d-cone, one of which  
% The angle
                               % subtended by the region where the sheet
                                % lifts off the rim of the container is a
                                % universal constant (approx.~139$^\circ$),
                                % independent of the indentation, the
                                % thickness and the material of the sheet (for
                                % small indentations, \cite{Cerda08032005}).
is that the tip of the d-cone consists of a crescent-shaped ridge where
curvature and elastic stress focus. In numerical simulations it was found that
the  radius of the crescent $R_{\rm cres.}$ scales with the thickness of the
sheet $h$ and the radius of the container $R_{\rm cont.}$ as $R_{\rm
  cres.}\sim h^{1/3}R_{\rm cont.}^{2/3}$. This dependence on the container
radius of the shape of the region near the tip  is not fully understood
\cite{RevModPhys.79.643}. As argued in this latter reference, it cannot be
explained by an analysis of the dominant contributions to the elastic energy,
which are: The bending energy from the region far away from the center, which
is well captured by modeling the d-cone as a developable surface there; and
the bending and stretching energy  part from a core region of size $O(h)$
where elastic strain is not negligible. The result of this (non-rigorous)
argument is an energy scaling $E\sim h^2 (C_1|\log h|+C_2)$. This is a natural
guess -- the situation here bears some resemblance to vortices in the
Ginzburg-Landau model, where this is the right
energy scaling \cite{MR1269538}. \\ 
In \cite{2012arXiv1208.4298M,BK1} the scaling of the elastic energy of a d-cone with respect to its thickness $h$ has been analyzed in a rigorous setting. The result from \cite{2012arXiv1208.4298M} is
\begin{equation}
h^2 (C_1|\log h|-C_2\log|\log h|)\leq {\mathcal E}_h \leq h^2 (C_1|\log h|+C_3)\,.
\label{eq:lowb}
\end{equation}
The lower bound does not achieve the conjectured scaling behaviour, and it seems that this claim can not be proved
with the methods used in \cite{2012arXiv1208.4298M}. \\
\\
Here we consider another situation which involves the regularization of
an isometric cone through the higher order bending energy. The main difference
with the (general) d-cone is that here the underlying cone is a surface of revolution. Hence it is meaningful to study the problem of the competition
between bending and stretching energies in a radially symmetric setting. This
makes it possible to use ODE methods in addition to energy methods.
We will show that a scaling result analogous to the
one above without the $\log\log h$ terms on the left hand side holds in this
simpler setting. \\
\\
% Crucial for a better understanding of the energy scaling in the
% two-dimensional model is some control over the rate at which the deformed
% configuration converges to the conical one when one goes in radial direction
% from the center of the disk towards its periphery. For general deformations of
% two dimensional sheets, it is
% very difficult to obtain this information.\\
% \\
% Here we consider a problem which is motivated by a different setting in which
% conical singularities arise.
% Here, we  
% investigate a model for which it is natural to consider only radially symmetric
% configurations. This simplification allows us to
% quantify the decay of the difference between minimizers of the elastic energy
% and conical deformations for large radii.\\
% \\
The setting is the following: to create an approximately conical deformation of an
elastic sheet, we cut out a sector of angle $\beta$ and glue the edges of this
sector back together. This situation has been investigated numerically in \cite{2006PhRvE..73d6604L},
where it is called ``regular cone''. 
In this situation, radially symmetric deformations are admissible -- in
contrast to the
case of the d-cone, where the boundary conditions are not radially symmetric.
To the best of our knowledge, it is not known whether the global minimizers of
the ``regular cone'' are
radially symmetric. We nonetheless believe that a careful study of minimizers
within the class of radially symmetric deformations will help to understand the
structure of local and global minimizers as well as the local and global
stability of possible radially symmetric minimizers.
Since we are interested in the  asymptotic behaviour, we consider  a sheet of
infinite radius with free boundaries (after suitable renormalization of the
energy, see below).\\
\\
% and show that minimizers of the elastic energy get arbitrarily close to the
% conical configuration for $r\rightarrow \infty$
% radially symmetric, and looks like a  developable surface far away from the
% center. In a small region near the center, the strain should be non-zero.\\
% \\
Apart from the restriction to radially symmetric configurations, we make one
more simplification in comparison to the situation in \cite{2012arXiv1208.4298M}: % of the setting 
%  have already mentioned the restriction to radially symmetric
% configurations. Secondly
We use the so-called von-K\'arm\'an approximation of non-linear elasticity
\cite{ciarlet1997theory,MR2210909}. This means that the out-of-plane component of the
deformation is supposed to be of the order  $\e\ll 1$, and the size of the removed
sector as well as  the in-plane deformation are of order $\e^2$. All terms in the elastic energy of order $\e^k$, $k>4$, and of order $h^2\e^k$, $k>2$ are discarded.\\
\\
% We want to argue that any radially symmetric solution is close to a
% cone away from the origin. , where $r$ is the distance from the origin in the reference configuration. \\
% Instead of imposing Dirichlet boundary conditions on the outer boundary of the sheet
% We are going to leave the boundary conditions at the outer radius of the circular sheet free -- in this sense we are considering a free boundary value problem. We expect that far away from the center,  
% Boundary conditions at $\infty$...\\
%\\
As we will explain in Section \ref{model}, these considerations lead to the
definition of the free elastic energy density
\begin{equation*}
\rel=  (\hat w^2-1+\hat
  u')^2+\left(\frac{\hat u}{r}\right)^2+\he^2\left(\hat
    w'^2+\frac{\hat w^2}{r^2}\right)
\end{equation*}
where $\he=h/\e$ and the deformation of the sheet is given as a map from spherical to
cylindrical coordinates by
\begin{equation}
(r,\varphi)\mapsto \left(r+\frac{\e^2}{2}(\hat u-r),\sqrt{1+\e^2}\varphi,\e W\right)\label{eq:12}
\end{equation}
with $W'=\hat w$. The renormalized energy functional is
\begin{equation}
\begin{array}{rrl}\hat E_\he:& \W& \to \R\cup\{+\infty\}\\
&(\hat u,\hat w)&\mapsto
\lim_{R\to\infty}\int_0^R r \d r\left( \rel(r)-\he^2\frac{\psi(r/\he)^2}{r^2}\right)\end{array}
\label{introen}
\end{equation}
where 
\begin{align*}
 \W=\Big\{&(\hat u,\hat w)\in W^{1,2}_{\rm loc}((0,\infty),\R^2):\,\int_0^1 r \d r \rel(r)<\infty\Big\}\,,
\end{align*}
and $\psi$ is some cutoff
function with $\psi(r)=0$ for $r$ close to $0$ and $\psi(r)=1$ for $r\geq
1$. 
We will show in Lemma \ref{pro:boundary} that the condition $\int_0^1 r \d r
\rel(r)<\infty$ implies $\hat u(0)=\hat w(0)=0$ and thus the deformation
\eqref{eq:12} is continuous at the origin for all $(\hat u,\hat w)\in \W$.\\

% $\hat u:(0,\infty)\to\R$ can be thought of as the in-plane deformation (in radial direction) and  $\hat w:(0,\infty)\to\R$ as the derivative of the out-of-plane deformation, and $\psi$ is some arbitrarily chosen smooth function $[0,\infty)\to\R$ with 
% \\
The aim of the present contribution is to prove
\begin{theorem}
\label{mainthm}
The functional $\hat E_\he$ from eq.~\eqref{introen} is well defined and bounded
from below. It possesses  minimizers $(\hat u,\hat w)$ in $\W$ with $\hat w\geq
0$ and  $\hat E_\he(\hat u,\hat w)<\infty$.
Furthermore, each minimizer $(\hat u,\hat w)$ with $\hat w\geq 0$
 satisfies
\begin{align*}
\hat u(r)=&\frac{\he}{2r}+o(\exp(-\sigma\sqrt{r/\he}))\\
\hat w(r)=&1+o(\exp(-\sigma\sqrt{r/\he}))\,
\end{align*}
as $r\to \infty$ for any $\sigma<2$.
\end{theorem}

As a side product of the proof of Theorem \ref{mainthm}, we will get a lower
bound for the elastic energy when the radius of the elastic sheet in the
reference configuration is assumed to be finite. This lower bound is better than
the analogous one from eq.~\eqref{eq:lowb} in that the $\log\log$-terms are not
present. 
% In the latter work, we found that the elastic energy
% \[
% \begin{array}{rrcl}{\mathcal E}_h:&W^{2,2}(B,\R^3)&\to& \R\\
% &y&\mapsto&\int_B\left|\nabla y^T\nabla y-\text{Id}\right|^2+h^2\left|\nabla^2 y\right|^2\d x\end{array}
% \]
% with $B=\{x\in\R^2:|x|\leq 1\}$ satisfies the lower bound 
% \begin{equation}
% \label{eq:lowb}
% C_1|\log h|-C_2\log|\log h|\leq h^{-2}{\mathcal E}_h  
% \end{equation}
% for boundary conditions compatible with a conical deformation. % Here we have an
% analogous lower bound without the $\log\log$-terms.
% We set $\he=h/\e$ as a
% parameter.
To give an idea how this ``improved'' lower bound comes about, let 
\begin{equation}
I_\he=\int_0^1 r \d r\rel\,.\label{eq:11}
\end{equation}
The first step to establish the lower
bound in the present setting is the right renormalization of the elastic energy density.
We expect a logarithmic divergence in $\lambda$ of $\he^{-2}I_\he$ as $\he\to 0$. Thus we
make the replacement
\[
\rel(r)\to\rel(r)-\lambda^{2}\frac{\psi(r/\he)^2}{r^2}
\]
% where $\psi\in C^\infty([0,\infty))$ is some cutoff function that with $\psi=0$
% near $r=0$ and $\psi=1$ for $r\geq \lambda$.
The key step is now to find a change of variables that makes it obvious that
\begin{equation}
\label{eq:2}
\int_0^1 r\d r\left(\rel(r)-\he^{2}\frac{\psi(r/\he)^2}{r^2}\right)
\end{equation}
is bounded from below by some constant times $\lambda^2$.
As we will see in Section \ref{existence}, such a change of variables does
exist,
and will leave us only with
manifestly positive terms in the renormalized energy eq.~\eqref{eq:2} plus some
divergence-like term that will be estimated in a suitable manner in Proposition
\ref{cor:welldef}. Thus we get the sought-for lower bound
\begin{equation}
\he^{-2}I_{\he}\geq |\log \he|-C \,.
\label{eq:lowb2}
\end{equation}
% which is the analog of eq.~\eqref{eq:lowb} for the setting discussed here; no
% $\log\log$ terms are present.\\
This paper is organised as follows: In Section \ref{model}, we motivate and
define our model. In Section \ref{existence} we establish a lower bound for the
renormalized energy and prove the existence of
minimizers of the  elastic free energy functional.  In a
remark at the end of that section,
we will discuss a pathology of the model presented here. In section
\ref{decay1}, we use stable manifold theory to show that minimizers converge to the conical configuration
at infinity.\\
\\
{\bf Notation.} In this paper, the letter $C$ stands for numerical constants
that are independent of all the other variables. Its value may change within the
same equation. In section \ref{model}, we will choose a cutoff function
$\psi\in C^\infty([0,\infty))$, that we have already mentioned above. The
cutoff function $\psi$  will then be fixed for the rest of the
paper. We will not indicate the dependence of constants on this choice of
$\psi$. Whenever we speak of functions $f\in W^{1,2}_{\rm loc.}(I)$ for some $I\subset\R$, it will be tacitly
understood that we mean its continuous representative.

\section{The model}
\label{model}
{\bf Cutting out a sector, glueing the edges back together.} 
% Identify the two dimensional plane with the complex numbers. Let the argument
% function (imaginary part of the logarithm) be defined by choosing a branch cut
% on the negative real line, $\arg:\C\setminus (-\infty,0]\rightarrow
% (-\pi,\pi)$. Let $B_\beta$ be the plane with a sector of angle $\beta$ removed,
For small $\e>0$ let $\beta^{(\e)}$ be defined by $2\pi/(2\pi-\beta^{(\e)})=\sqrt{1+\e^2}$,
and let
\begin{equation*}
B^{(\e)} =\R^2\setminus\left\{(x_1,x_2):x_2<0<x_1, \, -\beta^{(\e)}<\arctan x_2/x_1<0\right\}
\end{equation*}
We define a deformation $y:B^{(\e)}\rightarrow\R^3$ whose image has
rotational symmetry. We are going to use cylindrical coordinates $(r,\varphi)$,
\begin{equation}
y(r,\varphi)=U(r)e_{r}^{(\e)}+V(r)e_z
\end{equation}
where
\begin{align*}
U:&\,[0,\infty)\rightarrow [0,\infty)\non\\
V:&\,[0,\infty)\rightarrow (-\infty,\infty)\non\\
e_r^{(\e)}&=(\cos \varphi^{(\e)},\sin\varphi^{(\e)},0)\non\\
\varphi^{(\e)}&=\frac{2\pi}{2\pi-\beta^{(\e)}}\,\varphi
\end{align*}
We calculate
\begin{align*}
\nabla y=&\left(U'e_r^{(\e)}+V'e_z\right)\otimes e_r
+\frac{\sqrt{1+\e^2}}{r}U\, e_\varphi^{(\e)}\otimes e_\varphi\\
\nabla y^T\nabla y-\Id=&\left((U')^2+(V')^2-1\right)e_r\otimes
e_r+\left(\frac{1+\e^2}{r^2}U^2-1\right) e_\varphi\otimes e_\varphi\\
\nabla^2 y=& \left(U''e_r^{(\e)}+V''e_z\right)\otimes e_r\otimes e_r
+\sqrt{1+\e^2}\left(\frac{U}{r}\right)'e_\varphi^{(\e)}\otimes
\left(e_\varphi\otimes e_r+e_r\otimes e_\varphi\right)\\
&+\left(\frac{U'}{r}-\frac{(1+\e^2)U}{r^2}\right)e_r^{(\e)}\otimes
e_\varphi\otimes e_\varphi
+\frac{V'}{r}e_{z}\otimes e_\varphi\otimes e_\varphi
\end{align*}
{\bf Definition of the elastic energy.} As a starting point, we choose the elastic energy density to be
\begin{align*}
\bar\rho_{\rm el.}=&\left|\nabla y^T\nabla y-\Id\right|^2+h^2|\nabla^2 y|^2
\end{align*}
where $h$ is a parameter for the thickness of the sheet under
consideration. The first term on the right hand side represents the stretching
energy density, the second one the bending energy density. This is a standard
ansatz for the elastic energy, for a justification see
e.g.~\cite{MR2358334}.\\
\\
{\bf Von-K\'arm\'an  ansatz, change of variables.} Now we make the ansatz
\begin{align}
U(r)=&r+O(\e^2)\non\\
V(r)=&O(\e)\,,\label{vKdef}
\end{align}
where $\e$ is some small parameter. The following changes of variables will be convenient,
\begin{align}
U(r)=&r+\frac{\e^2}{2}(\hat
u(r)-r)\non\\
V'(r)=&\e \hat w(r)\label{uhatvardef}
\end{align}
and, alternatively, 
\begin{align*}
U(r)&=r+\frac{\e^2}{2}\left(u(r)+\he^2\frac{\psi(r/\he)}{2r}-r\right)\\
V'(r)&=\e (\psi(r/\he)+w(r))\,,
\end{align*}
where $\he=h/\e$, $\psi\in C^\infty([0,\infty)$ with $\psi(r)=0$ near $r=0$ and
$\psi(r)=1$ for $r\geq 1$. \\
% The $(\hat u, \hat w)$ variables will be convenient for the analysis near
% $r=0$; the $(u,w)$ variables be convenient for $r\rightarrow
% \infty$.% \footnote{The functioning of the analytic tools we will deploy in our
  % proofs are crucially dependent on the right choice of variables. In
  % particular, the $(u,w)$
  % variables are chosen such that the associated Euler-Lagrange equations become
  % homogeneous. This enables us to use stable manifold theory for the analysis
  % of the asymptotic behaviour.}\\
% A deformation of finite elastic energy must satisfy the boundary conditions
% \begin{equation*}
% \hat u(0)=u(0)=\hat w(0)=w(0)=0\,.
% \end{equation*}
% Hence we will try to minimize the elastic energy in the space $\W=\{(\hat u,\hat
% w)\in W^{1,2}_{\rm loc}([0,\infty),\R^2):\hat u(0)=\hat w(0)=0\}$. Here, by
% $W^{1,2}_{\rm loc}([0,\infty))$, we denote the space
% \[
% \left\{f:f|_{(a,b)}\in W^{1,2}(a,b) \text{ for all }0\leq a<b<\infty\right\}\,.
% \]
% The required boundary conditions make sense
% by the continuous embedding $W^{1,2}(0,1)\to C^{0,1/2}([0,1])$.
A short computation shows that the elastic energy density is given by
\begin{align*}
\bar\rho_{\rm el.}=&\e^4 
\left((\hat w^2-1+\hat u')^2+\left(\frac{\hat u}{r}\right)^2+\he^2\left(\hat w'^2+\frac{\hat w^2}{r^2}\right)
\right)\\
&+O(\e^6)+O(\e^6\he^2)\,.
\end{align*}
We build our model only considering the leading terms in $\e$ for fixed $\he$, defining
\begin{align}
 \rel=&\lim_{\e\rightarrow 0} \e^{-4}\bar\rho_{\rm el.}\non\\
=& \left((\hat w^2-1+\hat u')^2+\left(\frac{\hat u}{r}\right)^2+\he^2\left(\hat w'^2+\frac{\hat w^2}{r^2}\right)\right)\,.
\label{energyden}
\end{align}

% $(\bar V'^2+2\bar U')$ will be called the radial stress, and $(1+2\bar U/r)$ will be called the hoop stress.
% We expect that these stresses tend to 0 as $r\to \infty$ which is the motivation to set
% \begin{equation*}
% \hat u:=r\left(1+2\frac{\bar U}{r}\right)
% \end{equation*}
% We added the factor $r$ in front of the hoop stress so that we can expect $\hat u(r)\to 0$ as $r\to \infty$ and at the same time $\hat u(0)=0$, which will allow us to define an associatied linear space.
% Thus the radial stress becomes
% \begin{equation*}
% \bar V'^2+2\bar U'=\bar V'^2-1+\hat u'\,.
% \end{equation*}
% Since the elastic energy density does not depend on $\bar V$ but only on $\bar V'$, we introduce $\hat w=\bar V'$.
{\bf Renormalization and rescaling.} The $(u,w)$ variables have been chosen such that we expect $u,u', w, w'$ to vanish as $r\to\infty$. By an inspection of the energy density $\rel(r)$ we expect  that the integral
\[
\int_0^R r\d r \rel(r)
\]
diverges logarithmically as $R\to\infty$. To have some hope of a meaningful
limit for $R\to\infty$,  we introduce the renormalized functional
\begin{equation}
\begin{array}{rrcl}\hat E^R_\he:&\W&\to& \R\non\\
&(\hat u,\hat w)&\mapsto& \int_0^R r\d r
\left(\rel(r)-\he^2\frac{\psi(r/\he)^2}{r^2}\right)\,.
\end{array}
% =&\int_0^R r\d r\left((\hat w^2-1+\hat u')^2+\left(\frac{\hat u}{r}\right)^2+\he^2\left(\hat w'^2+\frac{\hat w^2-\psi(\cdot/\he)^2}{r^2}\right)\right)
\label{Rfinhat}
\end{equation}
In the sequel, we  will set $\he\equiv 1$, and derive all results for this value
of $\he$. The general case can  be
recovered by the change of variables $r\to r/\lambda$, which we will do at the
very end.
% , and write 
%  the fact that $(\hat u,\hat w)$ will be a minimizer of the
% functional $\hat E_\he$ (defined in eq.~\eqref{introen}) if and only if $(\he\hat u(\cdot/\he),\hat
% w(\cdot/\he))$ is a  minimizer of the functional $\hat E_1$.
In fact the change of variable formula yields
\begin{equation}
  \label{eq:covhe}
  \hat E_\he^R(\hat u,\hat w)=\he^2\hat E_1^{R/\he}\left(\he^{-1}\hat u(\he\cdot),\hat
w(\he\cdot)\right)
\end{equation}
for any $\hat u,\hat w\in W^{1,2}_{\rm loc}$.\\
We will use the following notation:
% We collect some notation that we are going to use in the rest of this paper.
\begin{align*}
\hat E^R_1=&\hat E^R\,\non\\
% \hat E(\hat u,\hat w)=&\hat E^\infty(\hat u,\hat w)\,\label{renormhat}\\
E^R(u,w)=&\hat E^R(\hat u,\hat w)
% E(u,w)=&\hat E(\hat u,\hat w)\,.
\end{align*}
For the reader's convenience, we summarize some of the notation for future
reference (with $\lambda\equiv 1$): 
\renewcommand{\rel}{\rho^{\rm el.}}
\begin{equation}
\boxed{
\begin{aligned}
% \W=&\Big\{(\hat u,\hat w)\in W^{1,2}_{\rm loc}([0,\infty),\R^2):\,\hat u(0)=\hat
% w(0)=0\Big\}\\
  \W=&\Big\{(\hat u,\hat w)\in W^{1,2}_{\rm loc}((0,\infty),\R^2): E^1(u,w) < \infty  \Big\}\\
\psi\in &C^\infty([0,\infty))\text{ with }\\
&\psi(r)=0\text{ near }r=0\text{ and }\psi(r)=1\text{ for }r\geq 1\\
 u(r)=& \hat u(r)- \frac{\psi(r)}{2r}\\
 w(r)=& \hat w(r)- \psi(r)\\
\rel=&(\hat w^2-1+\hat u')^2+(\hat
  u/r)^2+\hat w'^2+r^{-2}\hat w^2\\
\hat E^R,E^R:&\W\to\R\\
\hat E^R(\hat u,\hat w)=&\int_0^R r\d r\left(\rel(r)-r^{-2}\psi^2\right)\\
E^R(u,w)=& \hat E^R(\hat u,\hat w)
\end{aligned}
}
\label{eq:13}
\end{equation}

% {\bf Comment SM 17.7. 13} Typo in formula  $\rho_{el}$ ? The term $- \psi^2/r^2$ should only appear
% in the renormalized energy.

% {\bf Comment, SM, 9.7.13}   $W^{1,2}_{loc}[0, \infty)$ is not quite the right space for $u$ and $w'$
% because we at best  control $\int_0^1  r \e r  (u'^2 + w'^2)$ and not the unweigthed $L^2$ norm of the derivative.\\
% We should use $u, w \in W^{1,2}_{loc}((0, \infty))$ plus boundedness of the energy.\\
% This boundedness of the energy actually shows that $w$ (and subsequently $u$) is continuous up to $0$
% and the limit at $0$ is zero. I have made the corresponding changes.

\begin{prop}  \label{pro:boundary}  If $(u,w) \in \W$ then
\begin{equation}
\lim_{r \to 0} u(r) = 0,  \quad \lim_{r \to 0} w(r) = 0.
\end{equation}
\end{prop}

To show this we recall the following result.

\begin{lemma}  \label{lem:estsupH1a}
\begin{itemize}
\item[(i)] Let $a \in \R$ and let $I = (-\infty, a)$ or $I=(a, \infty)$. If $g \in W^{1,2}(I)$ then
\begin{equation}   \label{eq:estsupH1a}
\sup_I g^2 \le  2 \|g\|_{L^2}  \, \|g'\|_{L^2} \le  \int_I \d t  (g^2 + g'^2)
\end{equation}
and 
\begin{equation}
\lim_{t \to - \infty} g(t) = 0 \quad \mbox{or} \quad \lim_{t \to  \infty} g(t) = 0, \quad \mbox{respectively}.
\end{equation}
\item[(ii)] If $r_0 \in (0, \infty)$ let $J=(0, r_0)$ or  $J = (r_0, \infty)$ and assume that $h \in W^{1,2}_{loc}(J)$ and
\begin{equation}
\int_J  r \d r  \left[  h'^2 + \frac{h^2}{r^2} \right] < \infty
\end{equation}
then
\begin{equation}
\sup_I h^2 \le 2  \|h\|_{L^2(I; \d r/ r )} \,    \|h'\|_{L^2(I; r \d r )} \le \int_J  r \d r  \left[  h'^2 + \frac{h^2}{r^2} \right] 
\end{equation}
and 
\begin{equation}
\lim_{r \to 0} h(r) = 0 \quad \mbox{or} \quad \lim_{r \to  \infty} h(r) = 0, \quad \mbox{respectively}.
\end{equation}
\end{itemize}
\end{lemma}

\begin{proof} Assertion (ii) follows from assertion (i) and the change of variables $r = e^t$. 
To prove (i) note that $(g^2)' = 2 g g'$ and thus by the fundamental theorem of calculus
\begin{equation}
\sup g^2 - \inf g^2 = \int_I   \d t   \,  2 | g g'|   \le \int_I \d t  (g^2 + g'^2).
\end{equation}
Moreover $\inf g^2 = 0$ since $g^2 \in L^1(I)$ and $I$ is unbounded. 
This proves \eqref{eq:estsupH1a}. Assume that $I = (-\infty, a)$.Then for any $t < a$ we
also have
\begin{equation}
\sup_{(-\infty, t)} g^2  \le  \int_{-\infty}^t   \d t (g^2 + g'^2)
\end{equation}
Now the right hand side goes to $0$ as $t \to - \infty$. Thus $\lim_{t \to - \infty} g(t) = 0$. 
The case $I = (a, \infty)$ is analogous. 
\end{proof}

\begin{proof}[Proof of Proposition \ref{pro:boundary}]
It follows from the condition $E^1(u,w) < \infty$ and Lemma \ref{lem:estsupH1a}  that $\sup_{(0,1)} |\hat w| < \infty$ and $\lim_{r \to 0} \hat w(r) = 0$.
This implies that $\int_0^1 r \d r  \hat u'^2  < \infty$. Hence another application of  Lemma \ref{lem:estsupH1a}  shows that
$\lim_{r \to 0} \hat u(r) = 0$. Since $u=\hat u$ and $w = \hat w$ in some interval $(0, r_0)$ the assertion of 
the proposition follows. 
\end{proof}

\section{Existence of  minimizers}
\label{existence}

We will show that for $(u,w) \in \W$ the limit $E(u,w) = \lim_{R \to \infty} E^R(u,w)$ exists in $(- \infty, \infty]$ and that
$E$ has a minimizer in $\W$. 

The main difficulty is that the renormalized energy density $\rel(r)-r^{-2}\psi^2$ is not pointwise
positive and therefore  it is not clear that $E^R(u,w)$ is bounded from below as $R \to \infty$. 
We thus proceed in several steps.
\begin{enumerate}
\item We first show that $\rel(r)-r^{-2}\psi^2$ can be rewritten as a pointwise positive term
plus  $ - \frac1r \left( \frac{u}{r} \right)' $ (plus a harmless explicit term with rapid decay at infinity).
The point is that  $ - \frac1r \left( \frac{u}{r} \right)'$ reduces to a boundary term when integrated against
the measure $r \d r$.  Thus $E^R(u,w)$ is bounded from below by the positive functional
\begin{alignat}1
E^{+,R}(u,w) :=  & \int_0^1  r \d r   \left[  \left(\hat w^2 - 1 + \hat u' \right)^2 +
 \frac{{\hat u}^2}{r^2} +  {\hat w}'^{2} + \frac{ {\hat w}^2}{r^2}
 \right] \\
&  +
\int_1^R  r \d r \left[  (2w +w^2 + u')^2  + \frac{u^2}{r^2} + w'^{2} \right].
\end{alignat}
up to the terms $u(1)$ and $u(R)/R$. 
\item The integrand in $E^+$ is a sum of positive terms but 
does not directly give bounds for $w$ and $u'$. We derive an interpolation 
inequality which allows us to estimate $g$ in a (weighted) $L^2$ space
if we control $(g+f')$  and $f$ and $g'$ in suitably weighted $L^2$ spaces
(this is essentially and interpolation between $H^1$ and $H^{-1}$, see Lemma \ref{lem:interpol}).

\item We would like to apply the interpolation inequality with $g = 2 w + w^2$ but to control
$g'$ in $L^2$ we need to control the $L^\infty$ norm of $w$.  On the other hand if we control
$g$ and $g'$ in $L^2$ then we control $g$, and hence $w$, in $L^\infty$. 
In Lemma \ref{lem:boundwg} we
show that one
can simultaneously bound the (weigthed)  $L^2$ norm  of $g$ and the $L^\infty$ norm 
of $w$. This also gives  enough control of $u'$ to deduce that $u(R)/ R^{1/2}$ is controlled
by $E^{+,R}(u,w)$.

\item We finally bound $u(1)$ by a sublinear expression in $E^{+,R}(u,w)$.
This allows us to absorb the boundary terms and to obtain a lower bound
\begin{equation}
E^R(u,w) \ge \frac12 E^{+,R}(u,w) - C,
\end{equation}
for $R \ge R_0$. From this it easily follows that the limit $\lim_{R \to \infty} E^R(u,w)$ exists
and moreover that this limit is finite if and only if $E^+(u,w)$ is finite, see 
Lemma \ref{cor:welldef}.
\item With these preparations we deduce the existence of minimizers in the usual way by the direct method of
the calculus of variations.
\end{enumerate}

We begin by rewriting the renormalized energy. We have  $\hat w = w + 1$ and $\hat u = u + \frac{1}{2r}$ for $r \ge 1$ 
and thus the renormalized energy  (cf.~eq.~\eqref{eq:13}) simplifies  for $r \ge 1$  to
\begin{align}   
& \rel(r)-\frac{1}{r^2}=
     \left(2w+w^2+u'-\frac{1}{2r^2}\right)^2+   \left(\frac{u+1/(2r)}{r}\right)^2+w'^2+\frac{2w+w^2}{r^2}\non
 \\ 
= &  \left(2w+w^2+u'\right)^2  - \frac{1}{r^2} (2 w + w^2 + u') + \frac{1}{4 r^4} + \left(\frac{u}{r}\right)^2 
+ \frac{u}{r^3} + \frac{1}{4 r^4} +  w'^2 + \frac{2w+w^2}{r^2}  \non \\
= &  \left(2w+w^2+u'\right)^2  + \left(\frac{u}{r}\right)^2  + w'^2 - \frac{u'}{r^2} +  \frac{u}{r^3} + \frac{1}{2 r^4}
\non \\ 
= & \left(2w+w^2+u'\right)^2+\left(\frac{u}{r}\right)^2+w'^2-\frac{1}{r}\left(\frac{u}{r}\right)'+\frac{1}{2r^4}
\label{eq:renorm_positive2}
%\label{rhopos}
\,.
\end{align}

\bigskip 
The first three terms in   \eqref{eq:renorm_positive2} are positive,  the last term is harmless (since it is integrable with respect to
the measure $r \d r$) and the last but one term produces a boundary term $- u(R)/R + u(1)$ when
integrated against $r \d r$.

\begin{definition}
\label{def:EE} For $0 \le a < 1 < R \le \infty$ define
\begin{align}
E^+(u,w; (a,1))  &:= \int_a^1  r \d r   \left[ ( \hat w^2 - 1 + \hat u')^2  + \frac{\hat u^2}{r^2} + \hat w'^2 + \frac{\hat w^2}{r^2} \right], 
\label{eq:Ea1} \\
E^+(u,w; (1, R)) &:= 
 \int_1^R r\d
r\left[(2w+w^2+u')^2+\frac{u^2}{r^2}+w'^2\right]\label{eq:E1R}\\
E^{+,R}(u,w) & := E^+(u,w; (0,1)) + E^+(u,w; (1, R))  \label{eq:E+R}  \\
E^+(u,w) & :=  E^{+,\infty}
\end{align}
where $\hat w(r) = w(r) +\psi(r)$, $\hat u(r) = u(r)  + \frac{\psi(r)}{2r}$.
\end{definition}
Note that by positivity of the integrand and the monotone convergence theorem the limit 
$\lim_{R \to \infty} E^{+,R}(u,w)$ exists in  $\R \cup \{\infty\}$
for all $(u,w) \in \W$ and
agrees with $E^+(u,w)$.

\bigskip 

From the formula   \eqref{eq:renorm_positive2} for the renormalized energy we get
\begin{align}
E^1(u,w) & = E^+(u,w; (0,1)) - \int_0^1 r \d r \, \frac{\psi^2}{r^2} 
\label{eq:E1renorm} \\
E^R(u,w) & = E^{+,R}(u,w) + u(1) - \frac{u(R)}{R} + \frac14 (1 - R^{-2}) - \int_0^1 r \d r \, \frac{\psi^2}{r^2}
\label{eq:ERrenorm}
\end{align}

We want to show that $E^R \ge \frac12 E^{+,R} - C$ for $R \ge R_0$ and that $\lim_{R \to \infty} u(R)/ R = 0$
if $E^+(u,w) < \infty$. One difficulty is that the functional $E^+(u,w)$ is not obviously coercive. 
We get immediately bounds for $u$ and $w'$, but there are no direct bounds for $u'$ and $w$. 
We will obtain bounds on $u'$ and $w$ from an interpolation result. 
To state it, it is more convenient to make the change of variables 
$R = e^T$, $\tilde u (t) = u(e^t)$ and $\tilde w(t) = w(e^t)$. Then
\begin{align} 
E^+(u,w;(1,R)) &= \int_1^R r\d r\left[(2w+w^2+u')^2+\frac{u^2}{r^2}+w'^2\right]   \non \\
& = \int_0^T \d t  \left[  \left( e^t (2 \tilde w + \tilde w^2) +  \tilde u' \right)^2 +  \tilde u^2 +\tilde w'^2  \right] 
 \label{eq:trafonew1}
\end{align}

\begin{lemma}
\label{lem:interpol} Let $T \in [1, \infty]$,  let $I = (0, T)$ and suppose that
$f , g\in W^{1,2}_{\rm loc}(I)$ and
\begin{equation}
  e^{t/2} g + e^{-t/2} f'  \in L^2(I), \quad f \in L^2(I), \quad g' \in L^2(I).
\end{equation}
Then
\begin{equation}
g \in L^2(I), \quad e^{-t} f' \in L^2(I)
\end{equation}
and
\begin{equation}  \label{eq:interpolation1}
\| g\|_{L^2}^2  + \|e^{-t} f'\|_{L^2}^2 \le C  \left( \|  e^{t/2} g+  e^{-t/2} f' \|_{L^2}^2  +  \|f \|_{L^2}^2  +  \|g'\|_{L^2}^2 \right).
\end{equation}
\end{lemma}

\paragraph{Remark.} For $T=\infty$ one can also prove a bound with the optimal exponential rate: the $L^2$ norms of $e^{t/2} g$
and $e^{-t/2} f'$ are bounded in terms by  a constant times $\|e^{-t/2} f' + e^{t/2} g\|_{L^2}  +  \|f \|_{L^2} +\|g'\|_{L^2}$.
To see this one 
uses the identity
$( e^{-t/2} f')^2 + ( e^{t/2} g)^2 = ( e^{-t/2} f' + e^{t/2} g)^2 - 2 f' g$ and integration by parts.
%and it integrate by parts (everything is well defined as $g \in L^2$). Now 
%\begin{equation}
%(f g)(0) = - \int_0^1 \d t \,   (f(t) g(t)  (1-t))'  =  \int_0^1  \d t  fg   + \int_0^1 \d t \,  f' g (1-t)  + \int_0^1 \d t \,   f g' (1-t). 
%\end{equation}
%Then one can use the $L^2$ bound for $g$ and the identity $f' = (f' + e^t g) - e^t g$ for $t \in [0,1]$.\\
The following example shows that one cannot estimate $e^{\beta t} g$ for any $\beta > \frac12$. 
Let  $\alpha > 0$ and
  $f = 2 e^{-\alpha t}  \sin e^{t/2}$, $g=  - e^{-\alpha t } e^{-t/2} \cos e^{t/2}$. Then $f$, $e^t g+f'$ and $g'$ are in $L^2((0, \infty))$
  but $e^{\beta t} g \notin L^2((0, \infty))$ if $\beta \ge \frac12 + \alpha$.

  \begin{proof} Let 
  \begin{equation}
G(t) := - \int_t^T  \d s \,  e^{-s/2} (e^{-s/2} f'(s) + e^{s/2} g)  +  \int_t^T \d s  \, e^{-s} f(s)   - e^{-t} f(t).    
\end{equation}
Note that both  integrals exist  (even for $T= \infty$) since $ e^{-s/2} f'(s) + e^{s/2} g \in L^2$, $f \in L^2$ and $e^{-s/2} \in L^2$. 
Moreover $G$ is absolutely continuous and for a.e.\  $t$ we have
\begin{equation}
G'(t) = e^{-t} f'(t) + g(t)  - e^{-t} f(t) -   (e^{-t} f(t))' =g(t).
\end{equation}
Thus $G \in W^{2,2}_{loc}(I)$ and $G'' = g'$.
Moreover by the Cauchy-Schwarz inequality
\begin{equation}
|G(t)| \le e^{-t/2}  \|e^{-s/2} f' + e^{s/2} g\|_{L^2} + e^{-t} \|f\|_{L^2} + e^{-t} f(t)
\end{equation}
and this implies that
\begin{equation}  \label{eq:GL2}
\|G \|_{L^2} \le   \|e^{-s/2} f'(s) + e^{s/2}\|_{L^2} + 2 \|f\|_{L^2}.
\end{equation}

For $a \in (0, T-1)$  we use the interpolation inequality
\begin{equation}
\|G' \|_{L^2((a, a+1)}^2 \le C  \left(  \|G\|_{L^2((a, a+1))}^2 + \|G''\|_{L^2((a,a+1))}^2 \right)
\end{equation}
For a proof see, e.g., ~\cite{gilbarg2001elliptic} or derive a contradiction from the 
assumptions $\|G'_j\|_{L^2((0,1))} = 1$ and $\|G_j\|_{L^2((0,1))}^2 + \|G''_j\|_{L^2((0,1))}^2 \to 0$. 
Passage to the limit $a \downarrow 0$ and $a \uparrow T-1$ (if $T < \infty$) shows that the inequality
also holds for $a= 0$ and $a=T-1$ (if $T < \infty$). If $T = \infty$ we sum the inequalities for $a \in \N$. 
If $T < \infty$ we denote by $[T]$ the integer part of $T$ and sum the inequalities for $a = 0, \ldots, [T] - 1$
and $a= T-1$. Since at most two of the intervals $(a, a+1)$ overlap we get
\begin{equation}
\|G' \|_{L^2}^2 \le  2 C  \left(   \|G\|_{L^2}^2 + \|G''\|_{L^2}^2 \right) 
\end{equation}
%Application of this inequality to the scaled function $G_a(t) = a G(t/a)$ and optimization in $a$ yields
%\begin{equation}
%\|G' \|_{L^2}^2 \le  2 C   \|G\|_{L^2}^2  \, \|G''\|_{L^2}^2
%\end{equation}
Since $G' =g$ and $G'' = g'$ the estimate for $\|g \|_{L^2}$ follows from \eqref{eq:GL2}.
The estimate for $e^{-t} f'$ follows from the triangle inequality since
$e^{-t} f'  = e^{-t/2} (e^{-t/2} f' + e^{t/2} g) - g$. 
  \end{proof}

We would like to apply the interpolation result with $f  = \tilde u$ and $g = 2 \tilde w + \tilde w^2$. 
We have $g' = 2 (1 + \tilde w) \tilde w'$ and $E^{+,R}$ controls only the $L^2$ norm of $\tilde w'$ 
and not directly the $L^2$ norm of $g'$. We thus simulataneously prove an $L^\infty$ bound for $\tilde w$
and an $L^2$ bound for $g$. 

\begin{lemma}  \label{lem:boundwg}
There exists a constant $C$ with the following property. If $R > 1$ and $(u,w) \in W^{1,2}_{\rm loc}([1, R))$ with
$E(R):=  E^+(u,w;(1,R)) < \infty$ then
\begin{align}
\sup_{[1, R]} |w|  & \le C (1 + E^{1/2}(R)), \label{eq:winfty} \\
\int_1^R  \frac{dr}{r}  \left[  (2 w + w^2)^2  + u'^2 \right]  & \le  C (1 + E(R)^2),
\label{eq:gL2} \\
R^{-1/2} |u(R)|  &  \le  C (1 + E(R)).  \label{eq:decayu}
\end{align}
\end{lemma}

\begin{proof} Let  $R = e^T$, $\tilde u(t) = u(e^t)$, $\tilde w(t) = w(e^t)$ and $g = 2 \tilde w + \tilde w^2$. 
To prove   \eqref{eq:winfty} we will assume in addition that $w \in L^\infty((1,R))$. This is no loss 
of generality since by the Sobolev embedding theorem $w \in L^\infty((1,R - \varepsilon))$ for all $\varepsilon$
positive. If we have  \eqref{eq:winfty} with  $R - \varepsilon$  instead of $R$ for all $\varepsilon > 0$ we can then consider the limit
$\varepsilon \downarrow 0$ to obtain the estimate for $R$. 

Let
\begin{equation}
M := \sup_{[0,T]} |\tilde w| = \sup_{[1,R]}|w|.
\end{equation}
If $M < 4$ there is nothing to show. We may thus assume $M \ge 4$. Then
\begin{equation}  \label{eq:boundgprime}
\frac12 M^2 \le \sup g,   \quad |g'| \le | 2 (1 + \tilde w) \tilde w'| \le 4 M |\tilde w'|.
\end{equation}
By  \eqref{eq:trafonew1} we have
\begin{equation}
E(R) =  \int_0^T \d t  \left[  \left( e^t g +  \tilde u' \right)^2 +  \tilde u^2 +\tilde w'^2  \right]. 
 \end{equation}
Thus arguing  as in Lemma \ref{lem:estsupH1a} and using the interpolation estimate with $f = \tilde u$  we get 
\begin{align}
\frac14 M^4 & \le \inf_{[0,T]} g^2 + ( \sup_{[0,T]} g^2 -  \inf_{[0,T]} g^2) 
\le \frac{1}{T}  \int_0^T \d t \, g^2  +     \int_0^T \d t \, g^2 +   \int_0^T \d t \, g'^2  \\
& \le 2 C  \int_0^T  \d t  \, \left[(e^{-t/2} \tilde u' + e^{t/2} g)^2 + \tilde u^2\right]  +   (2 C + 1)  \int_0^T \d t \, g'^2  \\
& \le  C  E(R)  + 16 M^2  (2 C+1)   E(R) \le \frac18 M^4   + C ( 1 +  E(R)^2),
\end{align}
where we  used Young's inequality $a b \le \frac18 a^2 + 2 b^2$. 
This implies    \eqref{eq:winfty}.

Now \eqref{eq:gL2} follows directly from the interpolation estimate, \eqref{eq:winfty} and 
 \eqref{eq:boundgprime}. Indeed we have
 \begin{align}
& \int_1^R   \frac{\d r}{r}  (2 w  + w^2)^2 = \int_0^T  \d t \, g^2  \non \\
 \le &  C  \int_0^T  \d t  \, \left[ (e^{-t/2} \tilde u' + e^{t/2} g)^2 + \tilde u^2  + g'^2 \right]  \non  \\
\le &  C (1 + E(R)) \,  E(R).
\end{align}
and the bound for $u'$ follows by the triangle inequality since $\int_1^R  r \d
r  (2 w + w^2  + u')^2 \le E(R)$ and $r^{-1}\leq r$ on $[1,\infty)$.

Using again the interpolation inequality and the $L^\infty$ bound for $w$ we get
\begin{align}
&  R^{-1} u^2(R) \le  \sup_{[0,T]}  e^{-t} \tilde u^2 
=  \inf_{[0,T]}   e^{-t} \tilde u^2 + ( \sup_{[0,T]}   e^{-t} \tilde u^2 -  \inf_{[0,T]} e^{-t} \tilde u^2)   \non \\
  \le &  \frac{1}{T}  \int_0^T  \d t \, e^{-t}  \tilde u^2 +   2 \int_0^T \d t  \,    \left( e^{-t} \tilde u^2 \right)'    \non \\
 \le &  3    \int_0^T  \d t \, e^{-t}  \tilde u^2  +     \int_0^T  \d t \,   \tilde u^2  +    \int_0^T  \d t \, e^{-2t}  \tilde u'^2  
  \le   C ( 1 +  E(R)^2).
\end{align}
Taking the square root we get \eqref{eq:decayu}.
\end{proof}

\begin{lemma}
\label{cor:welldef}
There exists a constant $C$  and $R_0 \ge 1$ such that for all $R \in [R_0, \infty)$ and all $(u,w) \in \W$ we have
\begin{equation}  \label{eq:lowerER}
E^R(u,w) \ge \frac12 E^{+,R}(u,w) - C.
\end{equation}
Moreover for all $(u,w) \in \W$ the limit
\begin{equation}
E(u,w) := \lim_{R \to \infty} E^R(u,w)
\end{equation}
exists in $\R \cup \{  \infty\}$ and
\begin{equation}
E(u,w) < \infty \quad \Longleftrightarrow \quad E^+(u,w) < \infty.
\end{equation}
In addition, if $E(u,w) < \infty$ then
\begin{equation}   \label{eq:identE+}
E(u,w) = E^+(u,w) + u(1) + \frac14 - \int_0^1r  \d r \frac{\psi^2}{r^2}.
\end{equation}
\end{lemma}

\begin{proof}
The starting point is the relation \eqref{eq:ERrenorm}
\begin{equation} \label{eq:ERrenorm2}
E^R(u,w)  = E^{+,R}(u,w) + u(1) - \frac{u(R)}{R} + \frac14 (1 - R^{-2}) - \int_0^1 r \d r \, \frac{\psi^2}{r^2}
\end{equation}
Note that with the notation of Lemma \ref{lem:boundwg} we
have
\begin{equation}
E^{+,R}(u,w) = E(R) + X_1, \quad \mbox{where   }  X_1 := E^+(u,w;(0,1)) \ge 0.
\end{equation}
By \eqref{eq:decayu}
\begin{equation}  \label{eq:estuR}
\frac{|u(R)|}{R} \le R^{-1/2} C (1 + E(R)) \le \frac14  E(R) + C
\end{equation}
if $R \ge R_0 := 4C$. 

Let $I=(0,1)$. By Lemma \ref{lem:estsupH1a}
\begin{align}
|\hat u(1)|^2 & \le  2 \| \hat u \|_{L^2(I;\d r/r)} \, \| \hat u' \|_{L^2(I; r \d r)} \\
&\le  2  \| \hat u \|_{L^2(I;\d r/r)}  \left(  
 \| \hat u' + \hat w^2 - 1\|_{L^2(I; r \d r)} + 
\|  \hat w^2 - 1 \|_{L^2(I; r \d r)}  \right)\\
& \le  X_1 + 2 X_1^{1/2}  \|  \hat w^2 - 1 \|_{L^2(I; r \d r)}.
\end{align}
Again by Lemma \ref{lem:estsupH1a} we have
$\sup_{[0,1]} \hat w^2 \le X_1$.  Thus
\begin{equation}
 \|  \hat w^2 - 1 \|_{L^2(I; r \d r)} \leq \sup |\hat w^2 - 1| \le (1 + X_1)
\end{equation}
and therefore $|\hat u(1)|^2 \le  (2 X_1^{1/2} + X_1 + 2  X_1^{3/2})$.
Using Young's inequality $ab \le \frac13 a^3 + \frac23 b^{3/2}$ first with
$(a,b)=(X_1^{1/2},1)$ and then with $(a,b)=(1,X_1)$
we get  $|\hat u(1)|^2 \le 4 (1 + X_1^{3/2})$.
 Finally we get
\begin{equation} \label{eq:estu1}
|u(1)|   = |\hat u(1) -\frac12|   \le  3  + 2  X_1^{3/4} \le 3 +  \frac14  8^4 + \frac{1}{4} X_1,
\end{equation}
where we used $a b \le \frac34 a^{4/3} + \frac14 b^4$ with $a = \frac14 X^{3/4} $ and $b = 8$.

Combining this with \eqref{eq:estuR} and  \eqref{eq:ERrenorm2}
and using that $X_1 \le E^{+,R}(u,w)$ and $E(R) \le E^{+,R}(u,w)$ we obtain
\eqref{eq:lowerER}  (for $R \ge R_0$).

Now if $E^+(u,w) = \infty$ then it follows from \eqref{eq:lowerER} that 
$\lim_{R \to \infty} E^{R}(u,w) = \infty$.  Assume now
$E^+(u,w) < \infty$. Since   $E^{+,R}(u,w) \le E^+(u,w)$ it follows from Lemma \ref{lem:boundwg}
that  $\lim_{R \to \infty} u(R)/R = 0$. In view of 
 \eqref{eq:ERrenorm2}  we deduce that that $\lim_{R \to \infty} E^{R}(u,w)$ exists and
 \begin{equation}
E(u,w) = E^+(u,w) + u(1) + \frac14 - \int_0^1r  \d r \frac{\psi^2}{r^2} < \infty.
\end{equation}
\end{proof}

% {\bf Comment SM 17.7, Corollary 1 has now been adapted to the new formulation}.

\begin{cor}
% Let $I_\lambda=\hat E_\lambda^1+\he^2\int_0^1 (\psi(r/\lambda)^2/r)\d r$.
For the unrenormalized energy $I_\he$ (cf.~eq.~\eqref{eq:11}), % There exist constants $C_1,C_2$ such that 
% the elastic energy defined in eq.~\eqref{Rfinhat} satisfies the energy scaling
\[
|\log \he|-C \leq \he^{-2}\inf I_\he \leq|\log \he|+C\,\quad\text{ for all }\lambda\in(0,R_0^{-1}).
\]
\end{cor}
\begin{proof}
By eq.~\eqref{eq:covhe}, 
\begin{equation}
  \label{eq:3}
\hat E_\he^1(\hat u,\hat w)=\he^2 \hat E^{\he^{-1}}_1(\hat u_\lambda, \hat w_\lambda)
\quad  \mbox{where}  \quad    \hat u_\lambda  = \he\hat u(\cdot/\he), \quad \hat w_\lambda  = \hat
w(\cdot/\he)\,.  
\end{equation}
Using the definition of $\hat E^R$, eq.~\eqref{eq:13}, we get
\[
\inf \hat E_\he^1\leq \he^2 \hat E^{\he^{-1}}(0,\psi)\leq C\he^2
\]
which proves the upper bound since
\[
I_\he=\hat E^1_\lambda+\he^2\int_0^1\psi(r/\lambda)^2\d r/r=\hat
E^1_\he+\he^2(C+|\log \he|)\,.
\] 
The lower bound follows from eq.~\eqref{eq:3} since by 
 \eqref{eq:lowerER} we have  
 \begin{equation}
 \hat E^{\he^{-1}}_1(\hat u_\lambda, \hat w_\lambda)
 = E^{\he^{-1}}(u_\lambda, w_\lambda) \ge -C
 \end{equation}
 for $\lambda \le 1/ R_0$.
\end{proof}

\renewcommand{\rel}{\rho^{\rm el.}}

Now we are in a position to prove the existence of minimizers for the
renormalized energy.

\begin{theorem}
\label{exist2}
We have $\inf_\W E \in \R$ and the functional $E$ attains its minimum in $\W$.
Moreover there exists a  minimizer $(u,w)$ of $E$ which satisfies
\begin{equation}  \label{eq:convergencew2}
w + \psi \ge 0 \quad \mbox{a.e.} % \quad \mbox{and} \quad \lim_{r \to \infty} w(r) = 0
\end{equation}
%or 
%\begin{equation}  \label{eq:convergencew3}
%w + \psi \le 0 \quad \mbox{a.e}  \quad \mbox{and} \quad \lim_{r \to \infty} w(r) = -2
%\end{equation}
\end{theorem}

\begin{proof} 
By Lemma \ref{cor:welldef}, $E$ is bounded from below.
Moreover $E(0,0) < \infty$. Thus $\inf E \in \R$.
Let $(u_j, w_j)$ be a  minimizing sequence, i.e.,
\begin{equation}
E(u_j, w_j) \to \inf_\W E.
\end{equation}
The energy is $E$ does not change if we replace $\hat w = w + \psi$ by $|\hat w|$. We may thus assume in addition that
\begin{equation}
w_j  + \psi \ge 0.
\end{equation}
Since $\sup_j E(u_j, w_j)$ is bounded we deduce from 
 \eqref{eq:lowerER}   that 
\begin{equation}
E^+(u_j, w_j)  \le C  \quad \forall j \in \N.
\end{equation}
Let $0 < a < b < \infty$. Then it follows directly from the formula for $E^+$
that $u_j$ and $w'_j$ are bounded in $L^2((a,b))$. By Lemma
\ref{lem:boundwg} the sequence $w_j$ is bounded in $L^\infty$. Thus $w_j$ and $w_j^2$ are  bounded in $L^2((a,b))$. 
Since $u'_j + 2w_j + w_j^2$
is bounded in $L^2((a,b))$ it follows that $u'_j$ is bounded in $L^2((a,b))$. 
Thus there exist a subsequence of $(u_j, w_j)$ which converges weakly in 
$W^{1,2}((a,b))$. We can apply this argument with $a = 1/k$, $b = k$ for $k \in \N$, $k \ge 2$
and successively select subsequences. By a diagonalization argument there exists a single subsequence
(still denoted by $(u_j, w_j)$) that converges weakly in $W^{1,2}_{loc}((0, \infty))$:
\begin{equation}
(u_j, w_j)  \rightharpoonup (u,w)   \quad \mbox{in  $W^{1,2}_{loc}((0, \infty))$. }
\end{equation}
By the compact Sobolev embedding this implies
\begin{equation}  \label{eq:locunif}
(u_j, w_j) \to (u,w)  \quad \mbox{locally uniformly in $(0, \infty)$}
\end{equation}
In particular we have the weak convergences
\begin{equation}
2 w_j + w_j^2 + u'_j  \rightharpoonup   2 w + w^2 + u    \quad \mbox{in $L^2_{loc}((0, \infty))$}
\end{equation}
and
\begin{equation}
\hat  w_j^2 - 1 + \hat u'_j  \rightharpoonup   \hat w^2 - 1 + \hat u    \quad \mbox{in $L^2_{loc}((0, \infty ))$,}
\end{equation}
where $\hat w_j = w_j + \psi$ , $\hat u_j = u_j + \psi/ 2r$, $\hat w = w + \psi$ , $\hat u = u + \psi/ 2r$.

Weak lower semicontinuity of the $L^2$ norm implies that for $0< a < 1 < b < \infty$.
\begin{align} 
& \int_a^1  r \d r    \left[ (\hat w^2 - 1 + \hat u'^2)^2 + \frac{\hat u^2}{r^2} + \frac{\hat w^2}{r^2} + \hat w'^2   \right]  \non \\
\le &
\liminf_{j \to \infty}
\int_a^1  r \d r    \left[ (\hat w_j^2 - 1 + \hat u_j'^2)^2 + \frac{\hat u_j^2}{r^2} + \frac{\hat w_j^2}{r^2} + \hat w_j'^2   \right] 
 \label{eq:lsc1}
\end{align}
and
\begin{align}  
& \int_1^b  r \d r    \left[ (2 w +  w^2  + u'^2)^2 + \frac{ u^2}{r^2}  + w'^2   \right]  \non \\
\le &
\liminf_{j \to \infty}
\int_a^1  r \d r    \left[ (2 w_j +  w_j^2 + u_j'^2)^2 + \frac{u_j^2}{r^2} + w_j'^2   \right].
\label{eq:lsc2}
\end{align}
Adding these two inequalities we get
\begin{equation}  \label{eq:lsc3}
E^+(u,w; [a,b]) \le \liminf_{j \to \infty} E^+(u_j, w_j),
\end{equation}
where $E^+(u,w;[a,b])$ is defined as the sum of the terms on the left hand side of
\eqref{eq:lsc1} and \eqref{eq:lsc2}. Finally the monotone convergence theorem 
implies that we can take the limit $a \to 0$ and $b \to \infty$ in \eqref{eq:lsc3}
and deduce
\begin{equation}
E^+(u,w) \le \liminf_{j \to \infty}  E^+(u_j, w_j). 
\end{equation}
This in particular implies that $E^+(u,w) < \infty$ and thus $(u,w) \in \W$. 

We now   use  the relation    \eqref{eq:identE+} between $E^+$ and $E$
and the fact that $u_j(1) \to u(1)$  (see \eqref{eq:locunif})
to deduce that
\begin{equation}
E(u,w) \le \liminf_{j \to \infty}  E(u_j, w_j) = \inf_W E. 
\end{equation}
Thus $(u,w)$ minimizes $E$ in $\W$.

Finally the condition $w_j + \psi \ge 0$ implies that $w + \psi \ge 0$.
%Hence $w \ge -1$ and Lemma \ref{cor:ww+2estim}
%implies that $\lim_{r \to \infty} w(r) = 0$. 
\end{proof}

\begin{rem}[Self-penetration of solutions]
The von K\'arm\'an model displays a pathology at the origin for the situation
we want to model. Namely, the solutions we have found show interpenetration of
matter. Consider again  the Euler-Lagrange equation obtained by variation of
$\hat u$,
\[
\left(r\left(\hat w^2-1+\hat u'\right)\right)'=\frac{\hat u}{r}\,.
\]
Since $\hat w\to 0$ for $r\to 0$, the qualitative behaviour of solutions $u$ near the
origin is the same as the one of solutions of the linear equation
\[
\left(r\left(\hat u'-1\right)\right)'=\frac{\hat u}{r}\,.
\]
The solutions of this latter equation are given by
\[
\frac{1}{2}r\log r+C_1r+C_2r^{-1}.
\]
The integration constant $C_2$ has to be set to zero to fulfill the boundary
condition $\hat u(0)=0$. Going back to eq.~\eqref{uhatvardef}, we see that the value of $U$ will be
negative in some punctured neighbourhood of the origin and we have
self-penetration of the solution (somewhere in the region $r\sim
\exp\left(-\e^{-2}\right)$). We expect that this pathology could be cured by
including nonlinear or higher order terms in $u$ in our model. We
refrain from doing so, since the main aspect of this work is the analysis of
the solutions away from the origin.
\end{rem}

\section{Decay properties}
% I: weighted energy estimates}
\label{decay1}

We now turn our interest to the decay properties of minimizers. We first  show
that $\lim_{r \to \infty} w(r) = 0$ (if $w + \psi \ge 0$).
%Then we show 
%that the decay of the energy density is fast in  a weighted $L^2$-sense. 
This will level the field for an application of stable manifold theory, by which we prove that $u$ and $w$ decay 
like a stretched exponential $\exp(- c \sqrt{r})$.

\begin{lemma}  
\label{cor:ww+2estim} Assume that  $(u,w)\in\W$
%% I have removed this extra condition for the time being, SM 9.7.13
%% $w\geq -\psi$ 
 with $E^+(u,w)<\infty$ and $w + \psi \ge 0$.
Then
\begin{equation}  \label{eq:convergencew}
\lim_{R \to \infty} w(R) = 0
\end{equation}
\end{lemma}

\begin{proof}
It follows from Lemma \ref{lem:boundwg} and Lemma \ref{lem:estsupH1a}
that
\begin{equation}
\lim_{r \to \infty} 2 w(r) + w^2(r) = 0.
\end{equation}

%To prove the last statement we apply Lemma \ref{lem:estsupH1a} 
%to $g = 2 \tilde w + \tilde w^2$ and we get
%\begin{equation}
%\lim_{t \to \infty} g(t) = 0.
%\end{equation}
Now  $ w(r) \ge -1$ and the function $F(s) = 2s + s^2 =(s+1)^2 - 1$ 
has a continuous inverse on $[-1, \infty)$. Thus $\lim_{r \to \infty}  w(r) =0$.
\end{proof}

\bigskip 

Now we show that minimizers actually have decay as $\exp(- c \sqrt r)$ at infinity.
We will use the following  standard tool from stable manifold theory:
\begin{theorem}[\cite{MR0069338}]
Let $s_0\in \R$, $A\in \R^{n\times n}$ a matrix with $k\leq n$ 
eigenvalues with negative real part and $n-k$ eigenvalues with positive real
part , $F:\R^n\times [ s_0,\infty)\to\R^n$
with the property that for every $\e>0$, there exist $S\in [ s_0,\infty)$ and
$\delta_0>0$ such that
\[
\left|F(x,s)-F(\bar x,s)\right|\leq \e |x-\bar x|
\]
whenever $|x-\bar x|\leq \delta_0$ and $s\geq S$.
Then there exist $\delta>0, \bar s>S$ such that for $|p|<\delta$ and
$\bar s>S$, there exists a $k$-dimensional submanifold $\bar M(\bar s)$ of $\R^n$ containing
the origin such that the initial value problem
\begin{equation}
\frac{\d}{\d s}x(s)=Ax(s)+F(x,s)\,,\quad x(\bar s)=p\label{1stsys}
\end{equation}
has a solution $x:[\bar s,\infty)\to\R^n$ for any $p\in \bar M(\bar s)$, and the property 
\[
|x(s)|=o(\exp(-\sigma s)) \text{ as }s\to\infty
\] 
for any $\sigma>0$ such that the absolute values of the real parts of the
eigenvalues of $A$ are all bigger than $\sigma$. Furthermore there exists
$\eta>0$ independent of $\bar s$ such that, if $p\not\in \bar M(\bar s)$, then 
\begin{equation*}
\|x\|_{L^\infty(s_0,\infty)}>\eta\,.
\end{equation*}
\label{stabmanif}
\end{theorem}

\begin{prop}
For any minimizer $(u,w)$ of $E$ with $w\geq-\psi$,
\begin{align*}
\left|\left(u(r),u'(r),w(r),w'(r)\right)\right|=o\left(\exp(-\sigma \sqrt{r})\right)
\end{align*}
as $r\to\infty$ for any $\sigma<2$.
\label{decayprop}
\end{prop}
\begin{proof}

% {\bf Comment 13.7.13 SM} I have added some motivation for the change of variables.

% {\bf SM 10.7.13} I found it easier to work with the functional for $\vu, \wu$  and then deduce the Euler Lagrange equations
% for $\vu, \wu$ from this.

%\bigskip

Since $E$ and $E^+$ only differ by a boundary term they lead to the same Euler-Lagrange equations. 
Thus for $r > 1$ the Euler Lagrange equations for $(u,w)$ are the same as the Euler-Lagrange equations for the functional
\begin{equation}
\int_1^\infty  r \d r \left [(2 w(r) + w^2(r) + u'(r))^2 + \frac{u^2(r)}{r^2} + w'^2(r) \right].
\end{equation}

It turns out that these EL equations are not of the form required in Theorem \ref{stabmanif} since the linear
part is not autonomous (up to a contribution which decays as $r \to \infty$). We will make a change of variables
to bring the EL equations in a suitable form. To motivate that change of variables it suffices to focus
on the linearization, i.e., we may neglect the terms $w^2$ in the energy functional (as we already know $w \to 0$
at $\infty$). The linearized equations are
\begin{align}
(r (2 w + u'))' & = \frac{u}{r} \\
2 r (2 w + u') & =  (r w')' = r w'' + w'
\end{align}
Differentiation of the second equation and use of the first yields  $r w''' + 2
w'' = 2 u/ r$. Thus $2 u' = r^2  w^{(4)}  + 4 r w^{(3)} + 2 w''$ and inserting this into the second  equation we
get the linearized fourth order equation for $w$
\begin{equation}
\frac{1}{2}(r^2  w^{(4)}  + 4 r w^{(3)} + 2 w'') + 2w = \frac{1}{2} w'' + \frac{1}{2r} w'
\end{equation}
Now we make the change of variables $w(r) = \wu(r^\alpha)$. Then
\begin{equation}
w' = \alpha r^{\alpha - 1} \wu', \quad w^{(k)} =   \alpha^k  r^{k (\alpha - 1)} \wu^{(k)} + \mbox{lower order derivatives}.
\end{equation}
This suggests to choose $\alpha = \frac12$ so that the leading order term in the linear equation becomes  
$\frac{1}{32} \wu^{(4)} + 2 w = 0$. We will now derive the EL equations in the new variables in 
detail. It is most convenient to first transform the functional.

We make the change of variables
\begin{equation}
w(r) = \wu(\sqrt{r}), \quad u(r) = \vu(\sqrt{r}),  \quad s = \sqrt{r}, \quad r = s^2.
\end{equation}
Then 
\begin{equation}
w'(r) = \frac{1}{2\sqrt{r}} \wu'(\sqrt{r}) = \frac{1}{2s} \wu'(s), \quad u'(r) = \frac{1}{2s} \vu'(s).
\end{equation}
Thus for $(u,w) \in \W$
\begin{align} \label{eq:trafo_energy}
E^+(u,w)  \ge &\int_1^\infty  r \d r \left [(2 w(r) + w^2(r) + u'(r))^2 + \frac{u^2(r)}{r^2} + w'^2(r) \right]  \non \\
= & \int_1^\infty 2 s^3 \d s \left[ (2\wu(s) + \wu^2(s) + \frac{1}{2s} \vu'(s) )^2  + \frac{\vu^2(s)}{s^4}  +  (\frac{1}{2s} \wu'(s))^2  \right] \non \\
%= &2   \int_1^\infty  \d s 
%\left[ s \left( s (2 \wu(s) + \wu^2(s)) + \frac12 \vu'(s) \right)^2
% + \frac{\vu^2(s)}{s} + \frac14  s \wu'^2(s) \right].\\
 = &2   \int_1^\infty  \d s 
\left[ s \left( s (2 \wu + \wu^2) + \frac12 \vu' \right)^2
 + \frac{\vu^2}{s} + \frac14  s \wu'^2 \right].
\end{align}

From \eqref{eq:trafo_energy} we easily obtain the Euler-Lagrange equations (for $s > 1$)
\begin{align}
2 s^2 (1 + \wu) \left(s (2 \wu + \wu^2) + \frac12 \vu'\right) &= \frac14   (s \wu')'   \label{eq:ELs1}\\
\frac12 \left[s \left(s (2 \wu + \wu^2) + \frac12 \vu' \right)\right]' = \frac{ \vu}{s}.  \label{eq:ELs2}
\end{align}
These equations first hold in the weak sense, but by standard elliptic regularity we get
$\wu \in W^{2,2}_{loc}$ and
 $\vu \in W^{2,2}_{loc}$ and the equations hold a.e. By induction one easily
sees that $(\vu, \wu)  \in W^{k,2}_{loc}$ 
for all $k$ and hence $(\vu, \wu) \in C^\infty$.

%
%
%{\bf Comment SM 10.7.13}: we should very briefly explain why the solutions
%are smooth and the EL equation hold in the strong form.\\
%I% have added some text.
%\bigskip

We choose $s_0>1$ large enough so that 
\begin{equation}
\frac{1}{2}<1+\wu(s)<\frac{3}{2} \text{ for }s\geq s_0\,,\label{lem1appls}
\end{equation} 
which is
possible by   \eqref{eq:convergencew}.
%Lemma \ref{w-1lem}(i).
 In this region we may divide eq.~\eqref{eq:ELs1}  by $2 s(1+w)$
and get
\begin{equation}   \label{eq:ELw2nd}
s \left(s (2 \wu + \wu^2) + \frac12 \vu' \right) = \frac{1}{8s} \frac{1}{(1+\wu)} (s \wu')'.
\end{equation}
Then \eqref{eq:ELs2} becomes
\begin{equation}  \label{eq:ELus}
\vu(s) = \frac{1}{2} s   \left[ \frac{1}{8s (1+\wu)} (s \wu')'\right]'.
\end{equation}
Inserting this into \eqref{eq:ELs1} we get a fourth order equation for $\wu$
\begin{equation}
(1+\wu)  \left[ s(2\wu +  \wu^2) + \frac{1}{4}   \left(s \left(\frac{(s \wu')'}{8s (1+\wu)} \right)'\right)' \right]
= \frac{1}{8s^2} (sw')'.
\end{equation}

This can be rewritten as
\begin{equation}
\wu^{(4)}(s)=-64 \wu(s)+g(x(s),s)+h(x(s),s)\label{4thmod}
\end{equation}
where $x(s)=(\wu^{(3)}(s),\wu''(s),\wu'(s),\wu(s))$, and $g:\R^4\times \R^+\to \R$ contains the nonlinear terms in $x$, $h:\R^4\times \R^+\to \R$ the linear ones with coefficients $O(s^{-1})$. More precisely,
\begin{align}
g(x,s)=&\frac{1}{1+\wu}\left(2\wu'\wu^{(3)}+\frac{4}{s}\wu''\wu'+\wu''^2-\frac{1}{s^2}\wu'^2-\frac{1}{1+\wu}\left(2\wu'^2\wu''+\frac{2}{s}\wu'^3\right)\right)\non\\
&-96\wu^2-32\wu^3\label{gdef}\\
h(x,s)=&-\frac{2}{s}\wu^{(3)}+\frac{5}{s^2}\wu^{(2)}+\frac{3}{s^3}\wu'\label{hdef}
\end{align}
In particular, for $f:=g+h$, and given $\e>0$, there exist $S\geq s_0$, $\delta>0$ such that
\begin{equation*}
\left|f(x,s)-f(\bar x,s)\right|\leq \e |x-\bar x|
\end{equation*}
whenever $|x-\bar x|<\delta$ and $s>S$. Additionally, we have $f(0,s)=0$. Now we may rewrite eq.~\eqref{4thmod} as a system of first order equations, 
\begin{equation*}
\frac{\d}{\d s}x(s)^T=Ax(s)^T+F(x,s)\,
\label{stabmanifappl}
\end{equation*}
where $F:\R^4\times [s_0,\infty) \to \R^4$ is given by $F(x,s)=(f(x,s),0,0,0)^T$ and
\begin{equation*}
A=\left(\begin{array}{cc} 0  & -64\\{\rm Id}_{3\times 3} & 0\end{array}\right)
\end{equation*}
The eigenvalues of $A$ are $2(\pm 1\pm i)$, i.e., $A$ has two
eigenvalues with positive real part and two with negative real part.

% {\bf SM 13.7.13} I get $2 (\pm 1 + \pm i)$. The eigenvalues of norm $2 \sqrt 2$. Hence their square
% as norm $8$ and their fourth power is $- 64$.

We already know that $\lim_{s \to \infty}\wu(s) = \lim_{s \to \infty} w(s^2) = 0$. We will now show that
\begin{equation}  \label{eq:fulldecay}
\lim_{s\to\infty}\wu^{(3)}(s)=\lim_{s\to\infty}\wu''(s)=\lim_{s\to\infty}\wu'(s)=0\,.
\end{equation}

It follows that $\lim_{s \to \infty}x(s) = 0$. From Theorem \ref{stabmanif}, it follows
that there exists $\bar s$ such that  $x(\bar s)\in \bar M(\bar s)$, and hence $|x(s)|=o(\exp(-\sigma s))$ for $\sigma<2$. 
It remains to prove \eqref{eq:fulldecay}.

We first  show $\wu'' \in L^2((s_0, \infty); \d s/s)$. From \eqref{eq:ELs1} we get
\begin{equation}  
| \wu'' + s^{-1} \wu'|   \le C s \left| s (2 \wu + \wu^2) + \frac12 \vu' \right|.
\end{equation}
Therefore
\begin{align*}  \label{eq:L2second}
 \int_{s_0}^\infty   \frac{\d s}{s} \wu''^2\     \leq  
2   \int_{s_0}^\infty   \frac{\d s}{s}  \left[  s^2 \left( s (2 \wu + \wu^2) + \frac12 \vu' \right)^2
+  s^{-2} \wu'^2 \right] 
 < \infty
\end{align*}
by   \eqref{eq:trafo_energy}.
Again by  \eqref{eq:trafo_energy} we have $\wu' \in L^2((s_0, \infty); s \d s)$. Thus 
Lemma \ref{lem:estsupH1a} yields
\begin{equation}  \label{eq:convergencew1}
\lim_{s \to \infty} \wu'(s) = 0.
\end{equation}

Next we derive a weighted $L^2$ estimate for the third derivative $\wu^{(3)}$.
It follows from  \eqref{eq:ELus} that
\begin{equation*}
\frac{\vu}{s}=\frac{1}{16(1+\wu)}\left(\wu^{(3)}+\frac{\wu''}{s}-\frac{\wu'}{s^2}-\frac{\wu'\wu''}{1+\wu}-\frac{\wu'^2}{s(1+\wu)}\right)\,,
\end{equation*}

% {\bf Comment SM 12.7.}: I changed $8(1 + \wu)$ into $16(1 + \wu)$. 

which implies (using %eqn.~\eqref{lem1appl} 
the convergence of $w$ and $w'$) %and \eqref{eq:decays})
\begin{equation}
\left|\wu^{(3)}(s)\right|^2\leq C\left(\left|\frac{\vu(s)}{s}\right|^2+\left|\frac{\wu''}{s}\right|^2+\left|\frac{\wu'}{s^2}\right|^2+  \wu''^2 \wu'^2+\left|\frac{\wu'^2}{s}\right|^2\right)\,.\label{w3rough}
\end{equation}
With the possible exception of $\wu''^2 \wu'^2$ all terms on the right hand side are integrable against $s \d s$. Thus we get for $s_1 > s_0$
\begin{equation}
\int_{s_0}^{s_1}  s \d s \left|   \wu^{(3)}\right|^2    \le C (1 + \sup_{[s_0, s_1]} |\wu''|^2),
\label{eq:1}
\end{equation}
where $C$ is controlled by $E^+(u,w)$ and in particular independent of $s_1$. 
Now we get as usual
\begin{align}
&\sup_{[s_0,S1]}  |\wu''|^2  - \inf_{[s_0,s_1]} |\wu''|^2 \le 2  \int_{s_0}^{s_1}  \d s   | \wu''   \wu^{(3)}|  \non \\
\le &  2    \| \wu'' \|_{L^2((s_0, \infty); \d s/ s)} C^{1/2}  ( 1 +  \sup_{[s_0, s_1]} |\wu''|^2)^{1/2}    \non \\
\le & 4 C   \| \wu'' \|_{L^2((s_0, \infty); \d s/ s)}^2 + \frac14  (1 +  \sup_{[s_0, s_1]} |\wu''|^2)
\label{eq:boundsupw2}
\end{align}
Moreover 
\begin{equation}
\inf_{[s_0, s_1]} \wu''^2  \le \frac{1}{\ln s_1/ s_0}    \int_{s_0}^{s_1} \frac{\d s}{s}    \wu''^2  \le \frac{C}{\ln s_1/ s_0}.
\end{equation}
Thus absorbing the term $ \frac14  (1 +  \sup_{[s_0, s_1]} |\wu''|^2)$ into the left hand side of the 
\eqref{eq:boundsupw2} and taking $s_1 \to \infty$ we get
\begin{equation}
\frac34 \sup_{[s_0,\infty]}  \wu''^2  \le  4 C   \| \wu'' \|_{L^2((s_0, \infty); \d s/ s)}^2 < \infty
\end{equation}
and by \eqref{eq:1}
\begin{equation}
\int_{s_0}^{\infty}  s \d s  \left| \wu^{(3)} \right|^2  < \infty.
\end{equation}
Since $\wu'' \in L^2((s_0, \infty); \d s/s)$ it follows that 
\begin{equation}
\lim_{s \to \infty} \wu''(s) = 0.
\end{equation}

Finally we claim that $\wu^{(4)} \in L^2 ((s_0, \infty); \d s/s )$.  Indeed from the previous 
$L^2$ bounds we see immediately that $h \in L^2 ((s_0, \infty); \d s/s) $. Moreover
the convergence of $\wu'$ and $\wu''$ imply that
$g(x(s),s) \le C (|w^{(3)}| + |\wu'| + |\wu''| + |\wu|)$.  By \eqref{eq:gL2}
 we have
\begin{equation}
\int_{s_0}^\infty  \frac{\d s}{s}  (2 \wu + \wu^2)^2  = 
\frac12  \int_{s_0^2}^\infty  \frac{\d r}{r} (2 w + w^2)^2 < \infty.
\end{equation}
Since $\frac52 <  2 +\wu(s) < \frac72$, we get 
$\wu \in L^2((s_0, \infty); \d s/s)$. Together with the weighted $L^2$ estimates
for $\wu^{(i)}$ for $i=1,2,3$ we get $\wu^{(4)} \in L^2 ((s_0, \infty); \d s/s )$.
In combination  with the estimate  $\wu^{(3)} \in L^2((s_0, \infty); s \d s)$  this implies that
\begin{equation}
\lim_{s \to \infty} \wu^{(3)}(s) = 0.
\end{equation}

Thus $\wu^{(i)}=o(\exp(-\sigma
s))$ as $s\to \infty$ for $i=0,1,2,3$ and all $\sigma<2$.
This implies
$u(r),u'(r),w(r),w'(r)=o(\exp(-\sigma\sqrt{r}))$ as $r\to \infty$ for all $\sigma<2$.
\end{proof}

\bigskip

\begin{proof}[Proof of Theorem \ref{mainthm}]
For $\he=1$, this follows from  Theorem \ref{exist2}  and Proposition
\ref{decayprop}. For $\he\neq 1$, we recall that by \eqref{eq:covhe}, we have $\hat E_\he^R(\hat u,\hat
w)=\he^2 \hat E_1^{R/\he^2}(\hat u,\hat w)$ and hence 
\[
\hat E_\he(\hat u,\hat
w)=\he^2\lim_{R\to\infty} \hat E_1^{R/\he}(\hat u_\he,\hat w_\he)=\he^2
E(u_\he,w_\he)
\]
for all $(\hat u,\hat w)\in \W$, where we used the notation $\hat
u_\he=\he^{-1}\hat u(\he\cdot)$, $\hat w_\he=\hat w(\he\cdot)$, $u_\he(r)=\hat
u_\he(r)-\psi( r)/(2r)$, $w_\he=\hat w_\he-\psi$ and Lemma
\ref{cor:welldef}. Now all statements follow from the case $\he=1$ already treated.
\end{proof}

\bibliographystyle{plain}
\bibliography{symcone}
\end{document}